\documentclass[12pt,fleqn]{article}

\usepackage{amsmath,amsthm,amssymb}
\usepackage{geometry} 
\usepackage[pdfpagemode=UseNone,pdfstartview=FitH]{hyperref}
\usepackage{xcolor}

\newcommand{\abs}[1]{\left|#1\right|}
\newcommand{\bdry}[1]{\partial #1}
\newcommand{\closure}[1]{\overline{#1}}

\newcommand{\dhalf}{\dfrac{1}{2}}
\newcommand{\dint}{\ds{\int}}
\newcommand{\dist}[2]{\text{dist}\, (#1,#2)}
\newcommand{\ds}[1]{\displaystyle #1}
\newcommand{\eps}{\varepsilon}
\newcommand{\Fucik}{Fu\v c\'\i k }
\newcommand{\half}{\frac{1}{2}}

\newcommand{\id}[1][]{id_{\, #1}}
\newcommand{\ip}[3][]{\left(#2,#3\right)_{#1}}
\newcommand{\norm}[2][]{\left\|#2\right\|_{#1}}
\renewcommand{\O}{\text{O}}
\renewcommand{\o}{\text{o}}
\newcommand{\PS}[1]{$(\text{PS})_{#1}$}
\newcommand{\pnorm}[2][]{\if #1'' \left|#2\right|_p \else \left|#2\right|_{#1} \fi}
\newcommand{\restr}[2]{\left.#1\right|_{#2}}
\newcommand{\seq}[1]{\left(#1\right)}
\newcommand{\set}[1]{\left\{#1\right\}}
\newcommand{\sip}[3][]{(#2,#3)_{#1}}
\newcommand{\snorm}[2][]{\|#2\|_{#1}}
\newcommand{\R}{\mathbb R}
\newcommand{\vol}[1]{\left|#1\right|}
\newcommand{\NN}{\mathbb N}

\newenvironment{enumroman}{\begin{enumerate}

}{\end{enumerate}}

\newtheorem{lemma}{Lemma}[section]

\newtheorem{theorem}[lemma]{Theorem}

\theoremstyle{definition}

\numberwithin{equation}{section}

\title{\bf Nonlocal critical growth elliptic problems with jumping nonlinearities\thanks{{\em MSC2020:} Primary: 47J30, 35R11,
Secondary: 35S15, 35A15.
\newline \indent\; {\em Key Words and Phrases:} nonlocal elliptic problems, critical growth, jumping nonlinearities, nontrivial solutions}}

\author{\bf Giovanni Molica Bisci\\
Dipartimento di Scienze Pure e Applicate (DiSPeA)\\
Universit\`{a} degli Studi di Urbino Carlo Bo\\
Piazza della Repubblica 13, 61029 Urbino, Pesaro e Urbino, Italy\\
\em giovanni.molicabisci@uniurb.it\\
[\bigskipamount]
\bf Kanishka Perera\\
Department of Mathematics\\
Florida Institute of Technology\\
Melbourne, FL 32901, USA\\
\em kperera@fit.edu\\
[\bigskipamount]
\bf Raffaella Servadei\\
Dipartimento di Scienze Pure e Applicate (DiSPeA)\\
Universit\`{a} degli Studi di Urbino Carlo Bo\\
Piazza della Repubblica 13, 61029 Urbino, Pesaro e Urbino, Italy\\
\em raffaella.servadei@uniurb.it\\
[\bigskipamount]
\bf Caterina Sportelli\\
Department of Mathematics and Statistics\\
University of Western Australia\\
35 Stirling Highway, Crawley WA 6009, Australia\\
\em caterina.sportelli@uwa.edu.au}
\date{}

\begin{document}

\maketitle

\hfill

\begin{abstract}
In this paper we study a nonlocal critical growth elliptic problem driven by the fractional Laplacian in presence of jumping nonlinearities.
In the main results of the paper we prove the existence of a nontrivial solution for the problem under consideration, using variational and topological methods and applying a new linking theorems recently got by Perera and Sportelli in \cite{PeSp}.

The existence results provided in this paper can be seen as the nonlocal counterpart of the ones obtained in \cite{PeSp} in the context of the Laplacian equations. In the nonlocal framework the arguments used in the classical setting have to be refined. Indeed the presence of the fractional Laplacian operator gives rise to some additional difficulties, that we are able to overcome proving new regularity results for weak solutions of nonlocal problems, which are of independent interest.

\end{abstract}

\section{Introduction}
Fractional and nonlocal operators appear naturally in many different fields, see, for instance, \cite{apli1, apli2, caff, Vas} and the references therein. This is one of reason why nonlocal fractional problems are widely studied in the current literature.

From a mathematical point of view, nonlocal fractional Laplacian problems (and their generalizations to integrodifferential operators) have been considered in many different contexts and different results got in the classical context of uniformly elliptic equations have been extended to this nonlocal framework,
see for instance the monograph~\cite{MR3445279} for superlinear and subcritical cases, critical ones, and many others.

Motivated by the recent paper \cite{PeSp}, where the authors studied a critical growth elliptic problem with jumping nonlinearities, in this paper we consider its nonlocal counterpart. Precisely, here we study the existence of nontrivial solutions to the following nonlocal critical growth elliptic problem with a jumping nonlinearity
\begin{equation} \label{1}
\left\{\begin{aligned}
(- \Delta)^s\, u & = bu^+ - au^- + |u|^{2_s^\ast - 2}\, u && \text{in } \Omega\\[10pt]
u & = 0 && \text{in } \R^N \setminus \Omega,
\end{aligned}\right.
\end{equation}
where $\Omega$ is a open bounded subset of $\R^N$ with Lipschitz boundary satisfying the exterior ball condition, $(- \Delta)^s$ is the fractional Laplacian operator defined on smooth functions by
\[
(- \Delta)^s\, u(x) = 2\, \lim_{\eps \searrow 0}\, \int_{\R^N \setminus B_\eps(x)} \frac{u(x) - u(y)}{|x - y|^{N+2s}}\, dy, \quad x \in \R^N,
\]
$s \in (0,1)$, $N > 2s$, $2_s^\ast = 2N/(N - 2s)$ is the fractional critical Sobolev exponent, $a, b > 0$, and $u^\pm = \max \set{\pm u,0}$ are the positive and negative parts of $u$, respectively.

When $a = b = \lambda$, problem \eqref{1} reduces to the Br{\'e}zis-Nirenberg problem for the fractional Laplacian
\begin{equation} \label{2}
\left\{\begin{aligned}
(- \Delta)^s\, u & = \lambda u + |u|^{2_s^\ast - 2}\, u && \text{in } \Omega\\[10pt]
u & = 0 && \text{in } \R^N \setminus \Omega.
\end{aligned}\right.
\end{equation}
It was shown in Servadei \cite{MR3089742,MR3237009} and Servadei and Valdinoci \cite{MR3060890,MR3271254} that this problem has a nontrivial solution in each of the following cases:
\begin{enumroman}
\item $N > 4s$ and $\lambda > 0$,
\item $N = 4s$ and $\lambda > 0$ is not a Dirichlet eigenvalue of $(- \Delta)^s$ in $\Omega$,
\item $2s < N < 4s$ and $\lambda > 0$ is sufficiently large.
\end{enumroman}
This extends to the fractional setting the well-known results of Br{\'e}zis and Nirenberg \cite{MR709644} and Capozzi et al.\! \cite{MR831041} for critical Laplacian problems (see also Ambrosetti and Struwe \cite{MR829403} and Costa and Silva \cite{MR1306583}).

The set $\Sigma((- \Delta)^s)$ consisting of points $(a,b) \in \R^2$ for which the problem
\begin{equation} \label{3}
\left\{\begin{aligned}
(- \Delta)^s\, u & = bu^+ - au^- && \text{in } \Omega\\[10pt]
u & = 0 && \text{in } \R^N \setminus \Omega
\end{aligned}\right.
\end{equation}
has a nontrivial solution is called the Dancer-\Fucik spectrum of $(- \Delta)^s$ in $\Omega$. Denoting by $\seq{\lambda_l}$ the sequence of Dirichlet eigenvalues of $(- \Delta)^s$, the spectrum $\Sigma((- \Delta)^s)$ contains the points $(\lambda_l,\lambda_l)$ since problem \eqref{3} reduces to the Dirichlet eigenvalue problem for $(- \Delta)^s$ when $a = b$. Moreover, it follows from the abstract results in Perera and Schechter \cite[Chapter 4]{MR3012848} that in the square
\[
Q_l = (\lambda_{l-1},\lambda_{l+1}) \times (\lambda_{l-1},\lambda_{l+1}),
\]
$\Sigma((- \Delta)^s)$ contains two strictly decreasing curves
\[
C_l : b = \nu_{l-1}(a), \qquad C^l : b = \mu_l(a),
\]
with
\[
\nu_{l-1}(a) \le \mu_l(a), \qquad \nu_{l-1}(\lambda_l) = \lambda_l = \mu_l(\lambda_l),
\]
such that the points in $Q_l$ that are either below the lower curve $C_l$ or above the upper curve $C^l$ are not in $\Sigma((- \Delta)^s)$, while the points between them may or may not belong to $\Sigma((- \Delta)^s)$ when they do not coincide (see Theorem \ref{Theorem 5}).

The main results of the paper are the following.

\begin{theorem} \label{Theorem 1}
Let $N > 2 \left(\sqrt{2} + 1\right) s$. If $(a,b) \in Q_l$ and
\[
b \ge \mu_l(a)
\]
for some $l \ge 2$, then problem \eqref{1} has a nontrivial solution.
\end{theorem}

\begin{theorem} \label{Theorem 2}
Let $N \ge 4s$. If $(a,b) \in Q_l$ and
\[
b < \nu_{l-1}(a)
\]
for some $l \ge 2$, then problem \eqref{1} has a nontrivial solution.
\end{theorem}

The classical linking arguments used to obtain nontrivial solutions of problem \eqref{2} in \cite{MR3089742,MR3237009,MR3060890} rely on the decomposition of $H^s_0(\Omega)$ into eigenspaces of $(- \Delta)^s$. These arguments are not suitable for proving Theorems \ref{Theorem 1} and \ref{Theorem 2} since problem \eqref{3} is nonlinear and therefore its solution set is not a linear subspace of $H^s_0(\Omega)$ when $(a,b) \in \Sigma((- \Delta)^s)$.

We will prove our existence results using an abstract existence result for problems with jumping nonlinearities and new linking theorems based on a nonlinear splitting of the underlying space that were recently proved by Perera and Sportelli \cite{PeSp}. We will recall these results in Section~\ref{sec:abstractsetting} (see Theorems \ref{Theorem 7} and \ref{Theorem 3}).

The proof of Theorems \ref{Theorem 1} and \ref{Theorem 2} are in the line of the results got in \cite{PeSp}, even if additional difficulties arise, due to the presence of the fractional nonlocal operator $(-\Delta)^s$. In order to overcome these difficulties we need to perform a subtle analysis and to prove some regularity results for the eigenfunctions of $(-\Delta)^s$ (see Lemma \ref{Lemma 3.1}) and for weak solutions of nonlocal problems (see Lemma \ref{lemmaLinfty}). These complications are not only technical, but reflect the nonlocality of the problem.

The paper is organized as follows. Section \ref{sec:abstractsetting} is devoted to the construction of the minimal and maximal curves of the Dancer-\Fucik spectrum in an abstract setting introduced in \cite{MR3012848} and to some abstract existence results of linking type got in \cite{PeSp}. In Section \ref{sec:functional} we introduce the functional setting, give the variational formulation of problem \eqref{1} and prove some regularity results for nonlocal problems. Finally in Section \ref{sec:proofth11} and Section \ref{sec:proofth12} we prove the main results of the paper, concerning the existence of a nontrivial solution for problem \eqref{1}.

\section{Abstract setting}\label{sec:abstractsetting}
This section is devoted to some preliminary results we need in order to prove our existence theorems for problem \eqref{1}. In particular in Subsection \ref{subsec:dancerfucik} we recall some properties of the Dancer-\Fucik spectrum in a suitable abstract setting, while Subsection \ref{subsec:abstractexistence}
deals with some critical points theorems in presence of a suitable linking geometry, recently got in \cite{PeSp}.

\subsection{Dancer-Fucik spectrum in an abstract setting}\label{subsec:dancerfucik}
First we briefly recall the construction of the minimal and maximal curves of the Dancer-\Fucik spectrum in an abstract setting introduced in Perera and Schechter \cite[Chapter 4]{MR3012848}.

Let $H$ be a Hilbert space with the inner product $\ip{\cdot}{\cdot}$ and the associated norm $\norm{\cdot}$. Recall that an operator $\varphi : H \to H$ is monotone if $\ip{\varphi(u) - \varphi(v)}{u - v} \ge 0$ for all $u, v \in H$ and that $\varphi \in C(H,H)$ is a potential operator if $\varphi = \Phi'$ for some functional $\Phi \in C^1(H,\R)$, called a potential for $\varphi$. Assume that there are positive homogeneous monotone potential operators $p, n \in C(H,H)$ such that
\[
p(u) + n(u) = u, \quad \ip{p(u)}{n(u)} = 0 \quad \forall u \in H.
\]
We use the suggestive notation $u^+ = p(u),\, u^- = - n(u)$, so that
\[
u = u^+ - u^-, \quad \sip{u^+}{u^-} = 0.
\]
This implies
\[
\norm{u}^2 = \snorm{u^+}^2 + \snorm{u^-}^2,
\]
in particular, $\snorm{u^\pm} \le \norm{u}$.

Let $\mathcal A$ be a self-adjoint operator on $H$ with the spectrum $\sigma(\mathcal A) \subset (0,\infty)$ and $\mathcal A^{-1}$ compact. Then the spectrum $\sigma(\mathcal A)$ of $\mathcal A$ consists of isolated eigenvalues $\lambda_l,\, l \ge 1$ of finite multiplicities satisfying $0 < \lambda_1 < \cdots < \lambda_l < \cdots$.
Moreover, $D = D(\mathcal A^{1/2})$ is a Hilbert space with the inner product
\[
\ip[D]{u}{v} = \sip{\mathcal A^{1/2}\, u}{\mathcal A^{1/2}\, v} = \ip{\mathcal Au}{v}
\]
and the associated norm
\[
\norm[D]{u} = \snorm{\mathcal A^{1/2}\, u} = \ip{\mathcal Au}{u}^{1/2}.
\]
We have
\[
\norm[D]{u}^2 = \ip{\mathcal Au}{u} \ge \lambda_1 \ip{u}{u} = \lambda_1 \norm{u}^2 \quad \forall u \in H,
\]
so $D \hookrightarrow H$, and the embedding is compact since $\mathcal A^{-1}$ is a compact operator.

Let $E_l$ be the eigenspace of $\lambda_l$ and $N_l$ and $M_l$ be defined as follows
\[
N_l = \bigoplus_{j=1}^l E_j, \qquad M_l = N_l^\perp \cap D.
\]
Then $D = N_l \oplus M_l,\, u = v + w$ is an orthogonal decomposition with respect to both $\ip{\cdot}{\cdot}$ and $\ip[D]{\cdot}{\cdot}$. Moreover,
\[
\norm[D]{v}^2 = \ip{\mathcal Av}{v} \le \lambda_l \ip{v}{v} = \lambda_l \norm{v}^2 \quad \forall v \in N_l
\]
and
\[
\norm[D]{w}^2 = \ip{\mathcal Aw}{w} \ge \lambda_{l+1} \ip{w}{w} = \lambda_{l+1} \norm{w}^2 \quad \forall w \in M_l.
\]
We assume that $w^\pm \ne 0$ for all $w \in M_1 \setminus \set{0}$.

The set $\Sigma(\mathcal A)$ consisting of points $(a,b) \in \R^2$ for which the equation
\begin{equation} \label{2.1}
\mathcal Au = bu^+ - au^-, \quad u \in D
\end{equation}
has a nontrivial solution is called the Dancer-\Fucik spectrum of $\mathcal A$. It is a closed subset of $\R^2$ (see \cite[Proposition 4.4.3]{MR3012848}). Since equation \eqref{2.1} reduces to $\mathcal Au = \lambda u$ when $a = b = \lambda$, $\Sigma(\mathcal A)$ contains the points $(\lambda_l,\lambda_l)$.

It is easily seen that $\mathcal A$ is a potential operator with the potential
\[
\half \ip{\mathcal Au}{u} = \half \norm[D]{u}^2.
\]
The potentials of $p$ and $n$ are
\[
\half \ip{p(u)}{u} = \half\, \snorm{u^+}^2, \quad \half \ip{n(u)}{u} = \half\, \snorm{u^-}^2,
\]
respectively (see \cite[Proposition 4.3.2]{MR3012848}). So solutions of equation \eqref{2.1} coincide with critical points of the $C^1$-functional $\mathcal I:D\to \R$ given by
\begin{equation}\label{I}
\mathcal I(u,a,b) = \norm[D]{u}^2 - a\, \snorm{u^-}^2 - b\, \snorm{u^+}^2, \quad u \in D,
\end{equation}
and $(a,b) \in \Sigma(\mathcal A)$ if and only if $\mathcal I(\cdot,a,b)$ has a nontrivial critical point.

Let
\[
Q_l = (\lambda_{l-1},\lambda_{l+1}) \times (\lambda_{l-1},\lambda_{l+1}), \quad l \ge 2.
\]
If $(a,b) \in Q_l$, then $\mathcal I(v + y + w,a,b)$, with $v + y + w \in N_{l-1} \oplus E_l \oplus M_l$, is strictly concave in $v$ and strictly convex in $w$, i.e., for $v_1 \ne v_2 \in N_{l-1},\, w \in M_{l-1}$,
\[
\mathcal I((1 - t)\, v_1 + t v_2 + w,a,b) > (1 - t)\, \mathcal I(v_1 + w,a,b) + t \mathcal I(v_2 + w,a,b) \quad \forall t \in (0,1),
\]
and for $v \in N_l,\, w_1 \ne w_2 \in M_l$,
\[
\mathcal I(v + (1 - t)\, w_1 + t w_2,a,b) < (1 - t)\, \mathcal I(v + w_1,a,b) + t \mathcal I(v + w_2,a,b) \quad \forall t \in (0,1)
\]
(see \cite[Proposition 4.6.1]{MR3012848}).

\begin{theorem}[{\cite[Proposition 4.7.1, Corollary 4.7.3, \& Proposition 4.7.4]{MR3012848}}] \hspace{-0.0001pt} \label{Theorem 4}
Let $(a,b) \in Q_l$.
\begin{enumroman}
\item There is a positive homogeneous map $\theta(\cdot,a,b) \in C(M_{l-1},N_{l-1})$ such that $v = \theta(w,a,b)$ is the unique solution of
    \[
    \mathcal I(v + w,a,b) = \sup_{v' \in N_{l-1}} \mathcal I(v' + w,a,b), \quad w \in M_{l-1}.
    \]
    Moreover, $\theta$ is continuous on $M_{l-1} \times Q_l$, $\theta(w,\lambda_l,\lambda_l) = 0$ for all $w \in M_{l-1}$, and $\mathcal I'(v + w,a,b) \perp N_{l-1}$ if and only if $v = \theta(w,a,b)$.
\item There is a positive homogeneous map $\tau(\cdot,a,b) \in C(N_l,M_l)$ such that $w = \tau(v,a,b)$ is the unique solution of
    \[
    \mathcal I(v + w,a,b) = \inf_{w' \in M_l}\, \mathcal I(v + w',a,b), \quad v \in N_l.
    \]
    Moreover, $\tau$ is continuous on $N_l \times Q_l$, $\tau(v,\lambda_l,\lambda_l) = 0$ for all $v \in N_l$, and $\mathcal I'(v + w,a,b) \perp M_l$ if and only if $w = \tau(v,a,b)$.
\end{enumroman}
\end{theorem}

Now, let $S = \set{u \in D : \norm[D]{u} = 1}$ be the unit sphere in $D$. Set
\begin{equation} \label{27}
n_{l-1}(a,b) = \inf_{w \in M_{l-1} \cap S}\, \mathcal I(\theta(w,a,b) + w,a,b)
\end{equation}
and
\begin{equation} \label{28}
m_l(a,b) = \sup_{v \in N_l \cap S}\, \mathcal I(v + \tau(v,a,b),a,b).
\end{equation}
Since $\mathcal I(u,a,b)$ is nonincreasing in $a$ for fixed $u$ and $b$, and in $b$ for fixed $u$ and $a$, $n_{l-1}(a,b)$ and $m_l(a,b)$ are nonincreasing in $a$ for fixed $b$, and in $b$ for fixed $a$. Moreover, $n_{l-1}$ and $m_l$ are continuous on $Q_l$ and $n_{l-1}(\lambda_l,\lambda_l) = 0 = m_l(\lambda_l,\lambda_l)$ (see \cite[Lemma 4.7.6 \& Proposition 4.7.7]{MR3012848}). For $a \in (\lambda_{l-1},\lambda_{l+1})$, set
\begin{equation} \label{29}
\nu_{l-1}(a) = \sup \set{b \in (\lambda_{l-1},\lambda_{l+1}) : n_{l-1}(a,b) \ge 0}
\end{equation}
and
\begin{equation} \label{30}
\mu_l(a) = \inf \set{b \in (\lambda_{l-1},\lambda_{l+1}) : m_l(a,b) \le 0}.
\end{equation}

\begin{theorem}[{\cite[Theorem 4.7.9]{MR3012848}}] \label{Theorem 5}
Let $(a,b) \in Q_l$.
\begin{enumroman}
\item $\nu_{l-1}$ is a continuous and strictly decreasing function, $\nu_{l-1}(\lambda_l) = \lambda_l$, $(a,b) \in \Sigma(\mathcal A)$ if $b = \nu_{l-1}(a)$, and $(a,b) \notin \Sigma(\mathcal A)$ if $b < \nu_{l-1}(a)$.
\item $\mu_l$ is a continuous and strictly decreasing function, $\mu_l(\lambda_l) = \lambda_l$, $(a,b) \in \Sigma(\mathcal A)$ if $b = \mu_l(a)$, and $(a,b) \notin \Sigma(\mathcal A)$ if $b > \mu_l(a)$.
\item $\nu_{l-1}(a) \le \mu_l(a)$.
\end{enumroman}
\end{theorem}

Thus,
\[
C_l : b = \nu_{l-1}(a), \qquad C^l : b = \mu_l(a)
\]
are strictly decreasing curves in $Q_l$ that belong to $\Sigma(\mathcal A)$. They both pass through the point $(\lambda_l,\lambda_l)$ and may coincide. The region $\set{(a,b) \in Q_l : b < \nu_{l-1}(a)}$ below the lower curve $C_l$ and the region $\set{(a,b) \in Q_l : b > \mu_l(a)}$ above the upper curve $C^l$ are free of $\Sigma(\mathcal A)$. They are the minimal and maximal curves of $\Sigma(\mathcal A)$ in $Q_l$ in this sense. Points in the region $\set{(a,b) \in Q_l : \nu_{l-1}(a) < b < \mu_l(a)}$ between $C_l$ and $C^l$, when it is nonempty, may or may not belong to $\Sigma(\mathcal A)$.

\subsection{Abstract existence results}\label{subsec:abstractexistence}
In this subsection we recall the abstract existence results for problems with jumping nonlinearities got by Perera and Sportelli in \cite{PeSp}, that we will use to prove Theorems \ref{Theorem 1} and \ref{Theorem 2}.

Consider the equation
\begin{equation} \label{22}
\mathcal Au = bu^+ - au^- + f(u), \quad u \in D,
\end{equation}
where $a, b > 0$ and $f \in C(D,H)$ is a potential operator. Let $F \in C^1(D,\R)$ be the potential of $f$ that satisfies $F(0) = 0$, i.e.,
\[
F(u) = \int_0^1 \ip{f(su)}{u} ds, \quad u \in D
\]
(see \cite[Proposition 4.3.2]{MR3012848}). Solutions of equation \eqref{22} coincide with critical points of the $C^1$-functional $\mathcal E: D \to \R$ defined as follows
\[
\mathcal E(u) = \half \norm[D]{u}^2 - \frac{a}{2}\, \snorm{u^-}^2 - \frac{b}{2}\, \snorm{u^+}^2 - F(u) = \half\, \mathcal I(u,a,b) - F(u), \quad u \in D,
\]
where $\mathcal I$ is given in \eqref{I}.

We assume that
\begin{enumerate}
\item[$(F_1)$] $F(u) = \o(\norm[D]{u}^2)$ as $\norm[D]{u} \to 0$,
\item[$(F_2)$] $F(u) \ge 0$ for all $u \in D$,
\item[$(F_3)$] there exists $c^\ast > 0$ such that for each $c \in (0,c^\ast)$, every \PS{c} sequence of $\mathcal E$ has a subsequence that converges weakly to a nontrivial critical point of $\mathcal E$.
\end{enumerate}

Now we can state the following results.

\begin{theorem}[{\cite[Theorems 4.1 and 4.2]{PeSp}}] \label{Theorem 7}
Assume $(F_1)$--$(F_3)$, let $(a,b) \in Q_l$, and let $B = \set{v + \tau(v,a,b) : v \in N_l}$, where $\tau$ is given in Theorem \ref{Theorem 4}-$(ii)$.

Then equation \eqref{22} has a nontrivial solution in each of the following cases:
\begin{enumroman}
\item $b < \nu_{l-1}(a)$ and there exists $e \in D \setminus N_{l-1}$ such that
    \begin{equation} \label{20}
    \sup_{u \in Q}\, \mathcal E(u) < c^\ast,
    \end{equation}
    where $Q = \set{v + se : v \in N_{l-1},\, s \ge 0}$,
\item $b \ge \mu_l(a)$ and there exists $e \in D \setminus B$ with $-e \notin B$ such that
    \begin{equation} \label{21}
    \sup_{u \in Q}\, \mathcal E(u) < c^\ast,
    \end{equation}
    where $Q = \set{u + se : u \in B,\, s \ge 0}$.
\end{enumroman}
\end{theorem}

Finally, we state the linking theorem from Perera and Sportelli \cite{PeSp} that we will use to prove Theorem \ref{Theorem 2}. At this purpose, recall that a mapping $\varphi : Y \to Z$ between linear spaces is positive homogeneous if $\varphi(tu) = t \varphi(u)$ for all $u \in Y$ and $t \ge 0$ and denote by $\mathcal H$ the class of homeomorphisms $h$ of a Banach space $X$ onto itself such that $h$ and $h^{-1}$ map bounded sets into bounded sets.

\begin{theorem}[{\cite[Theorem 3.7]{PeSp}}] \label{Theorem 3}
Let $\mathcal E$ be a $C^1$-functional defined on a Banach space $X$ with norm $\|\cdot \|$. Let $X = N \oplus M$, with $N$ finite dimensional and $M$ closed and nontrivial.
Let $\theta \in C(M,N)$ be a positive homogeneous map and let $T : N \to X$ be a bounded linear map.

If $\norm{I_N - T}$ is sufficiently small, where $I_N$ is the identity map on $N$, and there exist $\rho > 0$ and $e \in X \setminus T(N)$ such that
\begin{equation} \label{7}
\sup_{u \in T(N)}\, \mathcal E(u) < \inf_{u \in A}\, \mathcal E(u), \qquad \sup_{u \in Q}\, \mathcal E(u) < \infty,
\end{equation}
where $S_\rho = \set{u \in X : \norm{u} = \rho}$, $A = \set{\theta(w) + w : w \in M \cap S_\rho}$ and $Q = \{u + se : u \in T(N),\, s \ge 0\}$, then
\begin{equation} \label{8}
\inf_{u \in A}\, \mathcal E(u) \le c := \inf_{h \in \mathcal{\widetilde{H}} }\, \sup_{u \in h(Q)}\, \mathcal E(u) \le \sup_{u \in Q}\, \mathcal E(u),
\end{equation}
where $\mathcal{\widetilde{H}} = \{h \in \mathcal H : \restr{h}{T(N)} = \id\}$.

Moreover, if $\mathcal E$ satisfies the {\em \PS{c}} condition, then $c$ is a critical value of $\mathcal E$.
\end{theorem}

\section{Functional setting and energy functional}\label{sec:functional}
In this section we introduce the functional setting and give the variational formulation of problem \eqref{1}. We also prove some regularity results useful in the sequel and we recall some properties of the best fractional Sobolev constant.

Let
\[
[u]_s = \left(\int_{\R^N\times \R^N} \frac{|u(x) - u(y)|^2}{|x - y|^{N+2s}}\, dx dy\right)^{1/2}
\]
be the Gagliardo seminorm of a measurable function $u : \R^N \to \R$ and let
\[
H^s(\R^N) = \set{u \in L^2(\R^N) : [u]_s < \infty}
\]
be the fractional Sobolev space endowed with the norm
\[
\norm[s]{u} = \left(\pnorm[2]{u}^2 + [u]_s^2\right)^{1/2},
\]
where $\pnorm[2]{\cdot}$ denotes the norm in $L^2(\R^N)$. We work in the closed linear subspace
\[
H^s_0(\Omega) = \set{u \in H^s(\R^N) : u = 0 \text{ a.e.\! in } \R^N \setminus \Omega},
\]
equivalently renormed by setting $\norm{\cdot} = [\cdot]_s$.

Problem \eqref{1} fits into the abstract setting of Subsection \ref{subsec:dancerfucik} with $H = L^2(\Omega)$, $p(u) = u^+$, $n(u) = - u^-$, $D = H^s_0(\Omega)$, and $\mathcal A$ equal to the inverse of the solution operator
$$\begin{aligned}
L^2(\Omega &) \to H^s_0(\Omega)\\
& g \mapsto u = ((- \Delta)^s)^{-1} g
\end{aligned}$$
of the problem
\[
\left\{\begin{aligned}
(- \Delta)^s\, u & = g(x) && \text{in } \Omega\\[10pt]
u & = 0 && \text{in } \R^N \setminus \Omega.
\end{aligned}\right.
\]
Since the embedding $H^s_0(\Omega) \hookrightarrow L^2(\Omega)$ is compact by \cite[Lemma~8]{svmountain} and \cite[Lemma~9]{MR3271254}, $\mathcal A^{-1}$ is compact on $L^2(\Omega)$.

Let $0 < \lambda_1 < \cdots < \lambda_l < \cdots$ be the sequence of distinct eigenvalues of the problem
\begin{equation} \label{8.5}
\left\{\begin{aligned}
(- \Delta)^s\, u & = \lambda u && \text{in } \Omega\\[10pt]
u & = 0 && \text{in } \R^N \setminus \Omega.
\end{aligned}\right.
\end{equation}
For a complete study of the eigenvalues and eigenfunctions of the fractional Laplace operator $(-\Delta)^s$ (and its generalization) we refer to \cite[Proposition~2.3]{MR3089742}, \cite[Proposition~9 and Appendix~A]{svlinking}, \cite[Proposition~4]{MR3060890} and \cite[Corollary~8]{servadeivaldinociREGO}.

Let $E_l$ be the eigenspace of $\lambda_l$, and let
\[
N_l = \bigoplus_{j=1}^l E_j.
\]

\subsection{Some regularity results}\label{subsec:regularity}
In this subsection we prove some regularity results useful in the sequel.

\begin{lemma} \label{Lemma 3.1}
We have $N_l \subset C^s(\R^N) \cap C^2(\Omega) \cap L^\infty(\R^N)$.
\end{lemma}

\begin{proof}
We have $N_l \subset L^\infty(\R^N)$ by Servadei and Valdinoci \cite[Proposition 4]{MR3060890} and taking into account that the eigenfunctions of $(-\Delta)^s$ vanish in $\R^N\setminus \Omega$. Then $N_l \subset C^s(\R^N)$ by Ros-Oton and Serra \cite[Proposition 1.1]{MR3168912} and \eqref{8.5}. Now iterating Ros-Oton and Serra \cite[Proposition 1.4]{MR3168912} shows that $N_l \subset C^2(\Omega)$. This ends the proof. 
\end{proof}

\begin{lemma} \label{lemmaLinfty}
Let $u\in H^s(\R^N)$ be a weak solution of the following problem
\begin{equation}\label{e+}
\left\{\begin{aligned}
(- \Delta)^s\, u & = f(x, u)\, && \text{in } \Omega\\[10pt]
u & = 0 && \text{in } \R^N \setminus \Omega,
\end{aligned}\right.
\end{equation}
where $f: \Omega \times \R \to \R$ is a Carath\'eodory function such that
\begin{equation}\label{fassumptio}
|f(x, t)|\leq \kappa|t|+g(x) \quad \mbox{a.e. } x\in \Omega\,\, \mbox{and for any } t\in \R
\end{equation}
for some $\kappa>0$ and $g\in L^\infty(\R^N)$.

Then~$u\in L^\infty(\R^N)$
and there exists~$C>0$, possibly depending on~$N$, $s$ and~$\kappa$,
such that
\begin{equation}\label{9:0}
\|u\|_\infty\le C |u|_2,.
\end{equation}
\end{lemma}

\begin{proof} We use arguments similar to the ones considered in \cite[Proposition 9]{servadeivaldinociREGO} and \cite[Proposition 4]{MR3060890}.

We may assume that $u$ does not vanish identically. Let~$\delta>0$,
to be taken appropriately small in what follows (the choice of~$\delta$
will be done on~\eqref{DELTA} below).
Up to multiplying~$u$ by a small constant, we may and do assume that
\begin{equation}\label{9:0:1}
|u|_2=\sqrt \delta\,.
\end{equation}

For any~$k\in\NN$, let~$C_k:=1-2^{-k}$, $A_k:=C_{k+1}/(C_{k+1}-C_k)=2^{k+1}-1$, $v_k:=u-C_k$, $w_k:= v_k^+$
and~$U_k:= |w_k|_2^2$.
It is easily seen that, by \cite[Lemma 5.2]{sv}, $w_k\in H^s(\R^N)$ and $u(x)-C_k=-C_k\leq 0$ a.e. $x\in \R^N\setminus \Omega$,
so that $w_k(x)=(-C_k)^+=0$
in~$\R^N\setminus \Omega$. Moreover, $0\le w_k\le |u|+|C_k|\le |u|+1\in L^2(\Omega)$
for any $k\in \NN$ and $w_k \to (u-1)^+\,\, \mbox{a.e. in } \R^N$.
As a consequence $w_k \to (u-1)^+$ in $L^2(\R^N)$ and so
\begin{equation}\label{8:0}
U_k=|w_k|_2^2\to |(u-1)^+|_2^2\,.
\end{equation}

Moreover, the sequence $w_k$ satisfies the following properties:
\begin{equation}\label{11:0}
w_{k+1}\le w_k\,\,\, \mbox{a.e. in}\,\, \R^N,
\end{equation}
\begin{equation}\label{set}
\{ w_{k+1}>0\}\subseteq \{ w_k>2^{-(k+1)}\}
\end{equation}
and
\begin{equation}\label{5:2}
\big| \{ w_{k+1}>0\} \big|\leq 2^{2(k+1)}U_k\,.
\end{equation}

Indeed, since $C_{k+1}>C_k$, it is easy to see that $v_{k+1}<v_k$ a.e. in $\R^N$, which yields \eqref{11:0}.

In order to show \eqref{set}, we
observe that if~$x\in \{w_{k+1}>0\}$ then $ 0<v_{k+1}^+(x)=\max\{u(x)-C_{k+1},0\}$.
Hence~$u(x)-C_{k+1}>0$ and $ v_k(x)=u(x)-C_k>C_{k+1}-C_k=2^{-(k+1)}$,
so that, as a consequence, $w_k(x)=v_k(x)>2^{-(k+1)}$\,, which proves~\eqref{set}.

Finally, to get \eqref{5:2}, it is enough to note that by \eqref{set} we obtain
$$ U_k =|w_k|_2^2 \ge\int_{\{ w_k>2^{-(k+1)}\}} w_k^2(x)\,dx
\ge 2^{-2(k+1)}\big| \{ w_k\ge 2^{-(k+1)}\} \big|\ge 2^{-2(k+1)} \big| \{ w_{k+1}>0\} \big|\,.
$$

Now, we prove that for any $k\in \NN$
\begin{equation}\label{99:0}
{\mbox{$|u|< A_k w_k$
on $\{w_{k+1}>0\}$\,.}}
\end{equation}
For this, let~$x\in \{w_{k+1}>0\}$\,. Then~$u(x)-C_{k+1}>0$ and so, by the properties of $C_k$, we have $u(x)>C_{k+1}>C_k\geq 0$\,. Hence, $w_k(x)=v_k(x)=u(x)-C_k$ and
$$\begin{aligned}
A_k w_k(x) & = A_k(u(x)-C_k)=\frac{C_{k+1}}{C_{k+1}-C_k}u(x)-
\frac{C_k C_{k+1}}{C_{k+1}-C_k}\\
& =u(x)+\frac{C_{k}}{C_{k+1}-C_k} \left(u(x)-C_{k+1}\right)> u(x)=|u(x)|,
\end{aligned}$$
which gives~\eqref{99:0}.

We also have that $v_{k+1}(x)-v_{k+1}(y)=u(x)-u(y)$ for any~$x$, $y\in\R^N$.
From this, \cite[Lemma 10]{servadeivaldinociREGO}, \eqref{e+}, the definition of $w_{k+1}$ and the fact that $u$ is a weak solution of \eqref{e+}, we deduce that
\begin{equation}\label{addadd}
\begin{aligned}
\|w_{k+1}\|^2 & =\int_{\R^N \times \R^N}
\frac{|v_{k+1}^+(x)-v_{k+1}^+(y)|^2}{|x\!-\!y|^{N+2s}}dxdy\\
& \le
\int_{\R^N \times \R^N}
\frac{\big(v_{k+1}(x)\!-\!v_{k+1}(y)\big)\big( v_{k+1}^+(x)\!-\!v_{k+1}^+(y)\big)}{|x\!-\!y|^{N+2s}}dxdy\\
& = \int_{\R^N \times \R^N}
\frac{\big(u(x)-u(y)\big)\big( v_{k+1}^+(x)-v_{k+1}^+(y)\big)}{|x-y|^{N+2s}}\,dx\,dy\\
& =\int_\Omega f(x,u(x))w_{k+1}(x)\,dx\\
& =\int_{\{w_{k+1}>0\}}f(x,u(x))w_{k+1}(x)\,dx\,.
\end{aligned}
\end{equation}
Therefore, recalling~\eqref{11:0} and~\eqref{99:0}, the definition of $A_k$ and $U_k$, \eqref{fassumptio} and H\"{o}lder inequality, by \eqref{addadd} we get
\begin{equation}\label{2:5}
\begin{aligned}
\|w_{k+1}\|^2 &
\le
\int_{\{w_{k+1}>0\}}|f(x, u(x))|w_{k+1}(x)\,dx\\
& \le \int_{\{w_{k+1}>0\}}\Big(\kappa|u(x)|+g(x)\Big)w_{k+1}(x)\,dx\\
& \le \kappa A_k\int_{\{w_{k+1}>0\}}w_k(x)w_{k+1}(x)\,dx + \|h\|_\infty \int_{\{w_{k+1}>0\}} w_{k+1}(x)\,dx\\
& \le \kappa A_k\int_{\{w_{k+1}>0\}}w_{k}^2(x)\,dx + \|g\|_\infty \int_{\{w_{k+1}>0\}} w_k(x)\,dx\\
& \le \kappa A_kU_k + \|g\|_\infty \big| \{ w_{k+1}>0\} \big|^{1/2}\int_{\{w_{k+1}>0\}} w_k^2(x)\,dx\\
& \le \kappa 2^{k+1}U_k + \|g\|_\infty \big| \{ w_{k+1}>0\} \big|^{1/2}U_k\,.
\end{aligned}
\end{equation}

Now we use the H\"{o}lder inequality and the fractional Sobolev inequality to get that
\begin{equation}\label{add}
\begin{aligned}
U_{k+1} &\le\left( \int_\Omega w_{k+1}^{2^*}(x)\,dx\right)^{2/2^*}
\big|  \{ w_{k+1}>0\} \big|^{2s/N}
\\ &\le
\tilde c \int_{\R^N \times \R^N}
\frac{|w_{k+1}(x)-w_{k+1}(y)|^2}{|x-y|^{N+2s}}\,dx\,dy
\;\big|  \{ w_{k+1}>0\} \big|^{2s/N}\,,
\end{aligned}
\end{equation}
for some positive constant $\tilde c$ depending only on $N$ and $s$\,.

Consequently, by~\eqref{5:2}, \eqref{2:5} and \eqref{add}, we get that
\begin{equation}\label{8:8}
\begin{aligned}
U_{k+1} & \le \tilde c \|w_{k+1}\|^2  \big|\{ w_{k+1}>0\} \big|^{2s/N}\\
& \le \tilde c \Big( \kappa 2^{k+1}U_k + \|g\|_\infty \big| \{ w_{k+1}>0\} \big|^{1/2}U_k \Big)\big|  \{ w_{k+1}>0\} \big|^{2s/N}\\
& \le \tilde c \Big( \kappa 2^{k+1}U_k + \|g\|_\infty \big| \{ w_{k+1}>0\} \big|^{1/2}U_k \Big)\left(2^{2(k+1)}U_k\right)^{2s/N}\\
& \le \tilde c \Big( \kappa 2^{k+1}U_k + \|g\|_\infty 2^{k+1}U_k^{3/2} \Big)2^{(4s/N)(k+1)}U_k^{2s/N}\\
& \le \left(1+c^*\right)2^{(4s/N+1)k}\Big(U_k^{1+2s/N}U_k^{3/2+2s/N}\Big)\\
& \le \Big(\left(1+c^*\right)2^{(4s/N+1)}\Big)^kU_k^{\gamma}\\
& = C^k U_k^{\gamma}\,,
\end{aligned}
\end{equation}
where $C=\left(1+c^*\right)2^{(4s/N+1)}$, with $c^*=\tilde c 2^{(4s/N+1)}\max\{\kappa, \|g\|_\infty\}$, and
$$\gamma=\begin{cases}
1+(2s/N)\,\,\, \mbox{if } U_k\leq 1\\
\frac 3 2+(2s/N)\,\,\, \mbox{if } U_k> 1\,.\\
\end{cases}$$

Now we are ready to perform our choice of~$\delta$: precisely, we assume that~$\delta>0$
is so small that
\begin{equation}\label{DELTA}
\delta^{\gamma-1}<\frac{1}{C^{1/(\gamma-1)}}
\end{equation}
We also fix $\eta\in \left(\delta^{\gamma-1},\frac{1}{C^{1/(\gamma-1)}}\right)$, so that, since $C>1$ and $\gamma>1$,
\begin{equation}\label{00:8}
\eta\in(0,1)\,.
\end{equation}
Moreover
\begin{equation}\label{00:8:1}
\delta^{\gamma-1}\le \eta\;{\mbox{ and }}\;C \eta^{\gamma-1}\le 1
\,.\end{equation}

Now, we claim that
\begin{equation}\label{00:9}
U_k \le\delta \eta^k\,.
\end{equation}
We proceed by induction and we observe that
$$ U_0=|u^+|_2^2\le|u|_2^2=\delta\,,$$
which is~\eqref{00:9} when~$k=0$. Let us now suppose that~\eqref{00:9}
holds true for~$k$ and let us prove it for~$k+1$. By \eqref{8:8}
and~\eqref{00:8:1} we deduce that
\begin{eqnarray*}
U_{k+1}\le C^k U_k^{\gamma}\le C^k (\delta \eta^k)^{\gamma}=\delta(C\eta^{\gamma-1})^k\delta^{\gamma-1}\eta^k
\le \delta\eta^{k+1}\,,
\end{eqnarray*}
which proves~\eqref{00:9}. Then, as a consequence of~\eqref{00:8} and~\eqref{00:9}, we get that
$\lim_{k\rightarrow+\infty}U_k=0$.
Thus, by~\eqref{8:0}, we have that $(u-1)^+=0$ a.e. in $\Omega$,
that is $u\le 1$ a.e. in $\Omega$.
Finally, by replacing $u$ with $-u$, we obtain $|u|_\infty\le 1$.
Then, \eqref{9:0} follows by recalling the scaling in \eqref{9:0:1}.
This completes the proof of Lemma~\ref{lemmaLinfty}\,.
\end{proof}

\subsection{Energy functional and best fractional critical Sobolev constant}\label{subsec:energyfunctional}
In this subsection we define the energy functional associated with problem~\eqref{1}. Later we recall some results related with the best fractional critical Sobolev constant.

The variational functional $\mathcal E: H^s_0(\Omega) \to \R$ associated with problem~\eqref{1} is defined by
\begin{multline*}
\mathcal E(u) = \half \int_{\R^N\times \R^N} \frac{|u(x) - u(y)|^2}{|x - y|^{N+2s}}\, dx dy - \half \int_\Omega \left[a\, (u^-)^2 + b\, (u^+)^2\right] dx - \frac{1}{2_s^\ast} \int_\Omega |u|^{2_s^\ast}\, dx\\[7.5pt]
= \half\, \mathcal I(u,a,b) - F(u)\,,
\end{multline*}
where
$$\mathcal I(u,a,b)=\int_{\R^N\times \R^N} \frac{|u(x) - u(y)|^2}{|x - y|^{N+2s}}\, dx dy - \int_\Omega \left[a\, (u^-)^2 + b\, (u^+)^2\right] dx$$
and the potential
\[
F(u) = \frac{1}{2_s^\ast} \int_\Omega |u|^{2_s^\ast}\, dx, \quad u \in H^s_0(\Omega).
\]

It is easily seen that $F$ clearly satisfies $(F_1)$ and $(F_2)$. It also follows from standard arguments that the energy functional $\mathcal E$ satisfies $(F_3)$ with
\[
c^\ast = \frac{s}{N}\, S_{N,\,s}^{N/2s},
\]
where
\begin{equation} \label{34}
S_{N,\,s} = \inf_{u \in H^s(\R^N) \setminus \set{0}}\, \frac{\dint_{\R^N\times \R^N} \frac{|u(x) - u(y)|^2}{|x - y|^{N+2s}}\, dx dy}{\left(\dint_{\R^N} |u|^{2_s^\ast}\, dx\right)^{2/2_s^\ast}}
\end{equation}
is the best fractional Sobolev constant.

The infimum in \eqref{34} is attained on the functions
\[
u_\eps(x) = c_{N,\,s} \left(\frac{\eps}{\eps^2 + |x|^2}\right)^{(N-2s)/2}, \quad \eps > 0,
\]
where the constant $c_{N,\,s} > 0$ is chosen so that
\[
\int_{\R^N\times \R^N} \frac{|u_\eps(x) - u_\eps(y)|^2}{|x - y|^{N+2s}}\, dx dy = \int_{\R^N} u_\eps^{2_s^\ast}\, dx = S_{N,\,s}^{N/2s}
\]
(see Servadei and Valdinoci \cite{MR3271254}).

Fix $x_0 \in \Omega$ and $\mu_0 > 2/\dist{x_0}{\bdry{\Omega}}$. Let $\xi : [0,\infty) \to [0,1]$ be a smooth function such that $\xi(t) = 1$ for $t \le 1/4$ and $\xi(t) = 0$ for $t \ge 1/2$. Set
\[
u_{\eps,\,\mu}(x) = \xi(\mu\, |x - x_0|)\, u_\eps(x - x_0), \quad \eps > 0,\, \mu \ge \mu_0.
\]

In the sequel we will apply Theorems \ref{Theorem 7} and \ref{Theorem 3} taking $e = u_{\eps,\,\mu}$ with $\eps > 0$ sufficiently small and $\mu \ge \mu_0$ sufficiently large. We have the following estimates for $u_{\eps,\,\mu_0}$ (see Servadei and Valdinoci \cite[Propositions 21 \& 22]{MR3271254}):
\begin{gather}
\label{37.1} \int_{\R^N\times \R^N} \frac{|u_{\eps,\,\mu_0}(x) - u_{\eps,\,\mu_0}(y)|^2}{|x - y|^{N+2s}}\, dx dy \le S_{N,\,s}^{N/2s} + c_1\, \eps^{N-2s},\\[15pt]
\label{45.1} \int_\Omega u_{\eps,\,\mu_0}^{2_s^\ast}\, dx \ge S_{N,\,s}^{N/2s} - c_2\, \eps^N,\\[15pt]
\label{38.1} \int_\Omega u_{\eps,\,\mu_0}^2\, dx \ge \begin{cases}
c_3\, \eps^{2s} - c_4\, \eps^{N-2s} & \text{if } N > 4s\\[7.5pt]
c_3\, \eps^{2s} \abs{\log \eps} - c_4\, \eps^{2s} & \text{if } N = 4s\\[7.5pt]
c_3\, \eps^{N-2s} - c_4\, \eps^{2s} & \text{if } 2s < N < 4s
\end{cases}
\end{gather}
for some constants $c_1, \dots, c_4 > 0$.

Noting that
\[
u_{\eps,\,\mu}(x) = \tilde{\mu}^{(N-2s)/2}\, u_{\tilde{\mu} \eps,\,\mu_0}(\tilde{\mu}\, x),
\]
where $\tilde{\mu} = \mu/\mu_0$, we get the following estimates for $u_{\eps,\,\mu}$ from \eqref{37.1}--\eqref{38.1}:
\begin{gather}
\label{37} \int_{\R^N\times \R^N} \frac{|u_{\eps,\,\mu}(x) - u_{\eps,\,\mu}(y)|^2}{|x - y|^{N+2s}}\, dx dy \le S_{N,\,s}^{N/2s} + c_5\, (\mu \eps)^{N-2s},\\[15pt]
\label{45} \int_\Omega u_{\eps,\,\mu}^{2_s^\ast}\, dx \ge S_{N,\,s}^{N/2s} - c_6\, (\mu \eps)^N,\\[15pt]
\label{38} \int_\Omega u_{\eps,\,\mu}^2\, dx \ge \begin{cases}
c_7\, \eps^{2s} - c_8\, \mu^{N-4s}\, \eps^{N-2s} & \text{if } N > 4s\\[7.5pt]
c_7\, \eps^{2s} \abs{\log\, (\mu \eps)} - c_8\, \eps^{2s} & \text{if } N = 4s\\[7.5pt]
c_7\, \mu^{N-4s}\, \eps^{N-2s} - c_8\, \eps^{2s} & \text{if } 2s < N < 4s
\end{cases}
\end{gather}
for some constants $c_5, \dots, c_8 > 0$. We will also need the following easily verified estimates (see \cite[Proposition 7.2]{MR3089742}):
\begin{gather}
\label{35} \int_\Omega u_{\eps,\,\mu}\, dx \le c_9\, \mu^{-2s}\, \eps^{(N-2s)/2},\\[15pt]
\label{36} \int_\Omega u_{\eps,\,\mu}^{2_s^\ast - 1}\, dx \le c_{10}\, \eps^{(N-2s)/2}
\end{gather}
for some constants $c_9, c_{10} > 0$.

\section{Proof of Theorem \ref{Theorem 1}}\label{sec:proofth11}
This section is devoted to the proof of Theorem \ref{Theorem 1}. At this purpose we will use the abstract linking result stated in Theorem \ref{Theorem 7}.

First of all, we note that, since $u^\pm = (-u)^\mp$, $u$ solves \eqref{3} (resp. \eqref{1}) if and only if $-u$ solves \eqref{3} (resp. \eqref{1}) with $a$ and $b$ interchanged. So $\Sigma(- \Delta)$ is symmetric about the line $a = b$ and we may assume without loss of generality that $a \le b$.

We start by proving some preliminary results.
For this let $S = \set{u \in H^s_0(\Omega) : \norm{u} = 1}$ and fix $\beta$ such that
\begin{equation} \label{46}
\frac{(N + 2s)\, s}{N} < \beta < \frac{N - 2s}{2}.
\end{equation}
This choice is admissible since $(N + 2s)\, s/N < (N - 2s)/2$, being $N > 2 \left(\sqrt{2} + 1\right) s$ by assumption.

We start with some preliminary results.

\begin{lemma} \label{Lemma 5}
Let $\beta$ as in \eqref{46} and let $K$ be a subset of $S \cap C^2(\Omega) \cap C(\closure{\Omega})$ such that
\begin{equation} \label{44}
\sup_{u \in K} \left(\norm[C^2(\closure{B_{1/\mu_0}(x_0)})]{u} + \norm[C(\closure{\Omega})]{u}\right) < \infty.
\end{equation}
Then there exist constants $c_{11}, \dots, c_{19} > 0$ such that for all $\eps > 0$, $\mu \ge \mu_0$, $u \in K$, and $\sigma, \tau \ge 0$,
\begin{multline} \label{39}
\int_{\R^N\times \R^N} \frac{|(\tau u + \sigma u_{\eps,\,\mu})(x) - (\tau u + \sigma u_{\eps,\,\mu})(y)|^2}{|x - y|^{N+2s}}\, dx dy \le \left(1 + c_{11}\, \mu^{-(N+2s)}\right) \tau^2\\[7.5pt]
+ \left(S_{N,\,s}^{N/2s} + c_{12}\, (\mu \eps)^{N-2s}\right) \sigma^2,
\end{multline}
\begin{multline} \label{40}
\int_\Omega |\tau u + \sigma u_{\eps,\,\mu}|^{2_s^\ast}\, dx \ge \left(\int_\Omega |u|^{2_s^\ast}\, dx - c_{13}\, \mu^{-N} - c_{14}\, \eps^{N\, [1 - 2 \beta/(N-2s)]}\right) \tau^{2_s^\ast}\\[7.5pt]
+ \left(S_{N,\,s}^{N/2s} - c_{15}\, (\mu \eps)^N - c_{16}\, \eps^{2N \beta/(N+2s)}\right) \sigma^{2_s^\ast},
\end{multline}
\begin{multline} \label{41}
\int_\Omega \left[a \left((\tau u + \sigma u_{\eps,\,\mu})^-\right)^2 + b \left((\tau u + \sigma u_{\eps,\,\mu})^+\right)^2\right] dx \ge \bigg(\int_\Omega \left[a\, (u^-)^2 + b\, (u^+)^2\right] dx\\[7.5pt]
- c_{17}\, \mu^{-4s}\bigg)\, \tau^2 + \Big(c_{18}\, \eps^{2s} - c_{19}\, \mu^{N-4s}\, \eps^{N-2s}\Big)\, \sigma^2.
\end{multline}
In particular,
\begin{multline*}
\mathcal E(\tau u + \sigma u_{\eps,\,\mu}) \le \half\, \Big(\mathcal I(u,a,b) + c_{20}\, \mu^{-4s}\Big)\, \tau^2 - \frac{1}{2_s^\ast}\, \bigg(\int_\Omega |u|^{2_s^\ast}\, dx - c_{13}\, \mu^{-N}\\[7.5pt]
- c_{14}\, \eps^{N\, [1 - 2 \beta/(N-2s)]}\bigg)\, \tau^{2_s^\ast} + \half \left(S_{N,\,s}^{N/2s} - c_{18}\, \eps^{2s} + c_{21}\, (\mu \eps)^{N-2s}\right) \sigma^2 - \frac{1}{2_s^\ast}\, \Big(S_{N,\,s}^{N/2s} - c_{15}\, (\mu \eps)^N\\[7.5pt]
- c_{16}\, \eps^{2N \beta/(N+2s)}\Big)\, \sigma^{2_s^\ast}
\end{multline*}
for some constants $c_{20}, c_{21} > 0$.
\end{lemma}

\begin{proof}
We have
\begin{multline} \label{42}
\hspace{-30pt} \int_{\R^N\times \R^N} \frac{|(\tau u + \sigma u_{\eps,\,\mu})(x) - (\tau u + \sigma u_{\eps,\,\mu})(y)|^2}{|x - y|^{N+2s}}\, dx dy = \tau^2\\[7.5pt]
\hspace{-30pt} + 2 \tau \sigma \int_{\R^N\times \R^N} \frac{(u(x) - u(y))(u_{\eps,\,\mu}(x) - u_{\eps,\,\mu}(y))}{|x - y|^{N+2s}}\, dx dy + \sigma^2 \int_{\R^N\times \R^N} \frac{|u_{\eps,\,\mu}(x) - u_{\eps,\,\mu}(y)|^2}{|x - y|^{N+2s}}\, dx dy
\end{multline}
since $u \in S$. By Molica Bisci et al.\! \cite[Lemma 1.26]{MR3445279},
\begin{multline} \label{42.3}
\int_{\R^N\times \R^N} \frac{(u(x) - u(y))(u_{\eps,\,\mu}(x) - u_{\eps,\,\mu}(y))}{|x - y|^{N+2s}}\, dx dy\\[7.5pt]
= - \int_{\R^N\times \R^N} u_{\eps,\,\mu}(x)\, \frac{u(x + z) - 2u(x) + u(x - z)}{|z|^{N+2s}}\, dz dx\\[7.5pt]
= - \int_{B_{1/2 \mu_0}} u_{\eps,\,\mu}(x) \left(\int_{\R^N} \frac{u(x + z) - 2u(x) + u(x - z)}{|z|^{N+2s}}\, dz\right) dx
\end{multline}
since $u_{\eps,\,\mu} = 0$ outside $B_{1/2 \mu_0}(x_0)$. For $z \in B_{1/2 \mu_0}(x_0)$, $x \pm z \in B_{1/\mu_0}(x_0)$ and hence a second-order Taylor expansion gives
\begin{equation} \label{42.1}
\abs{\int_{B_{1/2 \mu_0}} \frac{u(x + z) - 2u(x) + u(x - z)}{|z|^{N+2s}}\, dz} \le \norm[C(\closure{B_{1/\mu_0}(x_0)})]{D^2 u} \int_{B_{1/2 \mu_0}} \frac{dz}{|z|^{N-2\,(1-s)}}.
\end{equation}
On the other hand,
\begin{equation} \label{42.2}
\abs{\int_{B_{1/2 \mu_0}^c} \frac{u(x + z) - 2u(x) + u(x - z)}{|z|^{N+2s}}\, dz} \le 4 \norm[C(\closure{\Omega})]{u} \int_{B_{1/2 \mu_0}^c} \frac{dz}{|z|^{N+2s}}.
\end{equation}
The integrals on the right-hand sides of \eqref{42.1} and \eqref{42.2} are finite since $s \in (0,1)$. So combining \eqref{42.3}--\eqref{42.2} with \eqref{44} and \eqref{35} gives
\begin{equation} \label{43}
\int_{\R^N\times \R^N} \frac{(u(x) - u(y))(u_{\eps,\,\mu}(x) - u_{\eps,\,\mu}(y))}{|x - y|^{N+2s}}\, dx dy \le c_{11}\, \mu^{-2s}\, \eps^{(N-2s)/2}
\end{equation}
for some constant $c_{11} > 0$. Combining \eqref{42} with \eqref{43} and \eqref{37}, and noting that
\[
2 \tau \sigma \mu^{-2s}\, \eps^{(N-2s)/2} = 2 \mu^{-(N+2s)/2}\, \tau\, (\mu \eps)^{(N-2s)/2}\, \sigma \le \mu^{-(N+2s)}\, \tau^2 + (\mu \eps)^{N-2s}\, \sigma^2,
\]
we get \eqref{39} for some constant $c_{12} > 0$.

The elementary inequality
\[
|a + b|^p \ge |a|^p + |b|^p - p\, (p - 1)\, 2^{p-2} \left(|a|^{p-1}\, |b| + |a|\, |b|^{p-1}\right) \quad \forall a, b \in \R,\, p > 2
\]
together with \eqref{45}, \eqref{44}, \eqref{35}, and \eqref{36} gives
\begin{multline*}
\hspace{-4.2pt} \int_\Omega |\tau u + \sigma u_{\eps,\,\mu}|^{2_s^\ast}\, dx \ge \tau^{2_s^\ast} \int_\Omega |u|^{2_s^\ast}\, dx + \sigma^{2_s^\ast} \int_\Omega u_{\eps,\,\mu}^{2_s^\ast}\, dx\\[7.5pt]
\hspace{-4.2pt} - 2_s^\ast (2_s^\ast - 1)\, 2^{2_s^\ast - 2}\, \bigg(\tau^{2_s^\ast - 1} \sigma \int_\Omega |u|^{2_s^\ast - 1}\, u_{\eps,\,\mu}\, dx + \tau \sigma^{2_s^\ast - 1} \int_\Omega |u|\, u_{\eps,\,\mu}^{2_s^\ast - 1}\, dx\bigg)\\[7.5pt]
\hspace{-4.2pt} \ge \tau^{2_s^\ast} \int_\Omega |u|^{2_s^\ast}\, dx + \sigma^{2_s^\ast} \left(S_{N,\,s}^{N/2s} - c_6\, (\mu \eps)^N\right) - c_{13}\, \Big(\tau^{2_s^\ast - 1} \sigma \mu^{-2s}\, \eps^{(N-2s)/2} + \tau \sigma^{2_s^\ast - 1} \eps^{(N-2s)/2}\Big)
\end{multline*}
for some constant $c_{13} > 0$. Since
\[
\tau^{2_s^\ast - 1} \sigma \mu^{-2s}\, \eps^{(N-2s)/2} = \mu^{-(N+2s)/2}\, \tau^{2_s^\ast - 1}\, (\mu \eps)^{(N-2s)/2}\, \sigma \le \left(1 - \frac{1}{2_s^\ast}\right) \mu^{-N} \tau^{2_s^\ast} + \frac{1}{2_s^\ast}\, (\mu \eps)^N \sigma^{2_s^\ast}
\]
and
\[
\tau \sigma^{2_s^\ast - 1}\, \eps^{(N-2s)/2} = \eps^{(N-2s)/2 - \beta}\, \tau \eps^\beta\, \sigma^{2_s^\ast - 1} \le \frac{1}{2_s^\ast}\, \eps^{N\, [1 - 2 \beta/(N-2s)]}\, \tau^{2_s^\ast} + \left(1 - \frac{1}{2_s^\ast}\right) \eps^{2N \beta/(N+2s)}\, \sigma^{2_s^\ast}
\]
by Young's inequality, \eqref{40} follows.

Since $u_{\eps,\,\mu} = 0$ outside $B_{1/\mu}(x_0)$,
\begin{multline*}
\int_\Omega \left[a \left((\tau u + \sigma u_{\eps,\,\mu})^-\right)^2 + b \left((\tau u + \sigma u_{\eps,\,\mu})^+\right)^2\right] dx = \tau^2 \int_{\Omega \setminus B_{1/\mu}} \left[a\, (u^-)^2 + b\, (u^+)^2\right] dx\\[7.5pt]
+ \int_{B_{1/\mu}} \left[a \left((\tau u + \sigma u_{\eps,\,\mu})^-\right)^2 + b \left((\tau u + \sigma u_{\eps,\,\mu})^+\right)^2\right] dx = \tau^2 \int_\Omega \left[a\, (u^-)^2 + b\, (u^+)^2\right] dx\\[7.5pt]
+ \int_{B_{1/\mu}} \left[a \left((\tau u + \sigma u_{\eps,\,\mu})^-\right)^2 + b \left((\tau u + \sigma u_{\eps,\,\mu})^+\right)^2 - a\, (\tau u^-)^2 - b\, (\tau u^+)^2\right] dx.
\end{multline*}
Since $a \le b$, the last integral is greater than or equal to
\begin{multline*}
\int_{B_{1/\mu}} \left[a\, (\tau u + \sigma u_{\eps,\,\mu})^2 - b\, (\tau u)^2\right] dx = - (b - a)\, \tau^2 \int_{B_{1/\mu}} u^2\, dx + 2a \tau \sigma \int_{B_{1/\mu}}\! u u_{\eps,\,\mu}\, dx\\[7.5pt]
+ a \sigma^2 \int_{B_{1/\mu}} u_{\eps,\,\mu}^2\, dx \ge - c_{17}\, \tau^2\, \mu^{-N} - 2c_{18}\, \tau \sigma \mu^{-2s}\, \eps^{(N-2s)/2} + a \sigma^2\, \Big(c_7\, \eps^{2s} - c_8\, \mu^{N-4s}\, \eps^{N-2s}\Big)
\end{multline*}
for some constants $c_{17}, c_{18} > 0$ by \eqref{44}, \eqref{35}, and \eqref{38}. Since
\[
2 \tau \sigma \mu^{-2s}\, \eps^{(N-2s)/2} \le \mu^{-4s}\, \tau^2 + \eps^{N-2s}\, \sigma^2
\]
and $N > 4s$, \eqref{41} follows.
\end{proof}

Since $N > 2 \left(\sqrt{2} + 1\right) s$, $s^2/N < 1 - 2s/(N - 2s)$. Fix
\begin{equation} \label{50}
\frac{s^2}{N} < \gamma < 1 - \frac{2s}{N - 2s}
\end{equation}
and take $\mu = \eps^{- \gamma}$. Then $(\mu \eps)^{N-2s} = \eps^{(N-2s)(1 - \gamma)}$ and $(N - 2s)(1 - \gamma) > 2s$ by \eqref{50}, so it follows from Lemma \ref{Lemma 5} that for all small $\eps > 0$, $u \in K$, and $\sigma, \tau \ge 0$,
\begin{multline} \label{47}
\mathcal E(\tau u + \sigma u_{\eps,\,\eps^{- \gamma}}) \le \half\, \Big(I(u,a,b) + c_{20}\, \eps^{4 s \gamma}\Big)\, \tau^2 - \frac{1}{2_s^\ast}\, \bigg(\int_\Omega |u|^{2_s^\ast}\, dx - c_{13}\, \eps^{N \gamma}\\[7.5pt]
- c_{14}\, \eps^{N\, [1 - 2 \beta/(N-2s)]}\bigg)\, \tau^{2_s^\ast} + \half \left(S_{N,\,s}^{N/2s} - c_{22}\, \eps^{2s}\right) \sigma^2 - \frac{1}{2_s^\ast}\, \Big(S_{N,\,s}^{N/2s} - c_{15}\, \eps^{N (1 - \gamma)}\\[7.5pt]
- c_{16}\, \eps^{2N \beta/(N+2s)}\Big)\, \sigma^{2_s^\ast}
\end{multline}
for some constant $c_{22} > 0$. The proof of Lemma~\ref{Lemma 5} is complete.

\begin{lemma} \label{Lemma 6}
Let $K$ be a subset of $S \cap C^2(\Omega) \cap C(\closure{\Omega})$ such that \eqref{44} holds and
\begin{equation} \label{48}
\mathcal I(u,a,b) \le 0 \quad \forall u \in K.
\end{equation}
Then
\[
\sup_{u \in K,\, \sigma, \tau \ge 0}\, \mathcal E(\tau u + \sigma u_{\eps,\,\eps^{- \gamma}}) < \frac{s}{N}\, S_{N,\,s}^{N/2s}
\]
for all sufficiently small $\eps > 0$.
\end{lemma}

\begin{proof}
For $u \in K$, \eqref{48} together with the assumption that $a \le b$ and the H\"{o}lder inequality gives
\[
1 \le \int_\Omega \left[a\, (u^-)^2 + b\, (u^+)^2\right] dx \le b \int_\Omega u^2\, dx \le b \vol{\Omega}^{2s/N} \left(\int_\Omega |u|^{2_s^\ast}\, dx\right)^{2/2_s^\ast},
\]
so
\begin{equation} \label{49}
\inf_{u \in K}\, \int_\Omega |u|^{2_s^\ast}\, dx > 0.
\end{equation}
It follows from \eqref{47}--\eqref{49} that for all sufficiently small $\eps > 0$, $u \in K$, and $\sigma, \tau \ge 0$,
\begin{multline*}
\mathcal E(\tau u + \sigma u_{\eps,\,\eps^{- \gamma}}) \le \left[\half \left(1 - c_{23}\, \eps^{2s}\right) \sigma^2 - \frac{1}{2_s^\ast}\, \big(1 - c_{24}\, \eps^{N (1 - \gamma)} - c_{25}\, \eps^{2N \beta/(N+2s)}\big)\, \sigma^{2_s^\ast}\right] S_{N,\,s}^{N/2s}\\[7.5pt]
+ c_{26}\, \eps^{4 s \gamma}\, \tau^2 - c_{27}\, \tau^{2_s^\ast}
\end{multline*}
for some constants $c_{23}, \dots, c_{27} > 0$. Maximizing the right-hand side over all $\sigma, \tau \ge 0$ then gives
\[
\mathcal E(\tau u + \sigma u_{\eps,\,\eps^{- \gamma}}) \le \frac{s}{N}\, \frac{\left(1 - c_{23}\, \eps^{2s}\right)^{N/2s}}{\left(1 - c_{24}\, \eps^{N (1 - \gamma)} - c_{25}\, \eps^{2N \beta/(N+2s)}\right)^{(N-2s)/2s}}\, S_{N,\,s}^{N/2s} + c_{28}\, \eps^{2N \gamma/s}
\]
for some constant $c_{28} > 0$. Since $2N \beta/(N + 2s) > 2s$ by \eqref{46}, $N\, (1 - \gamma) > 2s$ and $2N \gamma/s > 2s$ by \eqref{50}, the desired conclusion follows.
\end{proof}

We are now ready to prove Theorem \ref{Theorem 1}.

\begin{proof}[Proof of Theorem \ref{Theorem 1}]
We apply Theorem \ref{Theorem 7} taking $e = u_{\eps,\,\eps^{- \gamma}}$ with $\eps > 0$ sufficiently small.

First of all we claim that $x_0 \in \Omega$ can be chosen so that $\pm u_{\eps,\,\eps^{- \gamma}} \notin B := \{v + \tau(v,a,b) : v \in N_l\}$ for $\varepsilon$ small enough. Suppose it is not the case. Then, also taking into account that $\Omega$ is an open set, there are sequences $\seq{x_j} \subset \Omega$ and $\eps_j \searrow 0$ such that $\eps_j^{- \gamma} > 2/\dist{x_j}{\bdry{\Omega}}$, the functions
\[
u_j(x) := u_{\eps_j,\,\eps_j^{- \gamma}}(x) = \xi(\eps_j^{- \gamma}\, |x - x_j|)\, u_{\eps_j}(x - x_j)
\]
have mutually disjoint supports, and either $\seq{u_j} \subset B$ or $\seq{- u_j} \subset B$. Suppose the former is the case, the latter case being similar. Let $\tilde{u}_j = u_j/\pnorm[2]{u_j}$. Then $\seq{\tilde{u}_j}$ is an orthonormal sequence in $L^2(\Omega)$. Moreover,  since $\tau(\cdot,a,b)$ is a positive homogeneous map, $\seq{\tilde{u}_j} \subset B$ and hence
\[
\tilde{u}_j = v_j + \tau(v_j,a,b)
\]
for some $v_j \in N_l$. Since $H^s_0(\Omega) = N_l \oplus M_l$ is an orthogonal decomposition with respect to the $L^2$-inner product, then
\[
\pnorm[2]{v_j}^2 + \pnorm[2]{\tau(v_j,a,b)}^2 = \pnorm[2]{\tilde{u}_j}^2 = 1
\]
and hence $\seq{v_j}$ is bounded in $L^2(\Omega)$. Since $N_l$ is finite dimensional, then a subsequence $\seq{v_{j_k}}$ converges in $L^2(\Omega)$ and in $H^s_0(\Omega)$ to some $v \in N_l$. Since $\tau(\cdot,a,b)$ is continuous from $L^2(\Omega)$ to $H^s_0(\Omega)$ (see the proof of \cite[Proposition
4.7.1]{MR3012848}), then $\tilde{u}_{j_k}$ converges to $v + \tau(v,a,b)$ in $H^s_0(\Omega)$ and hence also in $L^2(\Omega)$. This is a contradiction since $\seq{\tilde{u}_{j_k}}$ is an orthonormal sequence in $L^2(\Omega)$. Hence the claim is proved.

Now, to verify \eqref{21}, we use Lemma \ref{Lemma 6} with $K = S \cap B$. As in the proof of \cite[Theorem 4.2]{PeSp}, \eqref{48} holds. It only remains to show that $K \subset C^2(\Omega) \cap C(\closure{\Omega})$ and \eqref{44} holds.

At this purpose let $u \in K$. By Theorem \ref{Theorem 4}, $\mathcal I'(u,a,b) = z_u$ for some $z_u \in N_l$. Then
\begin{multline} \label{54}
\int_{\R^N\times \R^N} \frac{(u(x) - u(y))(\varphi(x) - \varphi(y))}{|x - y|^{N+2s}}\, dx dy - \int_\Omega \left(bu^+ - au^-\right) \varphi\, dx\\[7.5pt]
= \int_{\R^N\times \R^N} \frac{(z_u(x) - z_u(y))(\varphi(x) - \varphi(y))}{|x - y|^{N+2s}}\, dx dy \quad \forall \varphi \in H^s_0(\Omega),
\end{multline}
so $u$ is a weak solution of
\[
\left\{\begin{aligned}
(- \Delta)^s\, u & = bu^+ - au^- + (- \Delta)^s\, z_u && \text{in } \Omega\\[10pt]
u & = 0 && \text{in } \R^N \setminus \Omega.
\end{aligned}\right.
\]
Since $(- \Delta)^s\, (N_l) \subset N_l$, $(- \Delta)^s\, z_u = \tilde{z}_u$ for some $\tilde{z}_u \in N_l$, so we have
\begin{equation} \label{54.1}
\left\{\begin{aligned}
(- \Delta)^s\, u & = bu^+ - au^- + \tilde{z}_u && \text{in } \Omega\\[10pt]
u & = 0 && \text{in } \R^N \setminus \Omega.
\end{aligned}\right.
\end{equation}
Testing \eqref{54.1} with $\tilde{z}_u$ gives
\[
\pnorm[2]{\tilde{z}_u}^2 \le \norm{u} \norm{\tilde{z}_u} + \left(a\, |u^-|_2 + b\, |u^+|_2\right) \pnorm[2]{\tilde{z}_u},
\]
and since $K \subset S$ and $\norm{\tilde{z}_u} \le \sqrt{\lambda_l}\, \pnorm[2]{\tilde{z}_u}$, this implies that $\set{\tilde{z}_u : u \in K}$ is bounded in $L^2(\Omega)$. Since the set $\set{\tilde{z}_u : u \in K}$ is contained in the finite dimensional space $N_l$, which in turn is contained in $C^s(\R^N) \cap C^2(\Omega)\cap L^\infty(\R^N)$ by Lemma \ref{Lemma 3.1}, then it is also bounded in $C^s(\R^N) \cap C^2(\Omega)\cap L^\infty(\R^N)$.

Now, we apply Lemma \ref{lemmaLinfty} to problem \eqref{54.1} with $f(x,t)=bt^+ - at^- + g(x)$ and $g=|\tilde{z}_u|$. Since $\set{\tilde{z}_u : u \in K}$ is bounded in $L^\infty(\R^N)$ and
$$|f(x,t)|\leq \max\{|a|, |b|\}|t|+g(x)\,,$$
by Lemma \ref{lemmaLinfty} we get that $u\in L^\infty(\R^N)$ and so $K$ is bounded in $L^\infty(\R^N)$.

Then $K$ is also bounded in $C^s(\R^N)$ by Ros-Oton and Serra \cite[Proposition 1.1]{MR3168912}. Now iterating Ros-Oton and Serra \cite[Proposition 1.4]{MR3168912} shows that $K$ is contained in $C^2(\Omega)$ and bounded in $C^2(\closure{B_{1/\mu_0}(x_0)})$. Hence, $K \subset C^2(\Omega) \cap C(\closure{\Omega})$ and \eqref{44} holds. This concludes the proof of Theorem~\ref{Theorem 1}.
\end{proof}

\section{Proof of Theorem \ref{Theorem 2}}\label{sec:proofth12}
In this section we prove Theorem \ref{Theorem 2} using the abstract result stated in Theorem \ref{Theorem 3}.

Let $\eta : [0,\infty) \to [0,1]$ be a smooth function such that $\eta(t) = 0$ for $t \le 3/4$, $\eta(t) = 1$ for $t \ge 1$, and $|\eta'(t)| \le 5$ for all $t$. Set
\[
v_\mu(x) = \eta(\mu\, |x - x_0|)\, v(x), \quad v \in N_{l-1},\, \mu \ge \mu_0.
\]
We will apply Theorem \ref{Theorem 3} taking $T$ to be the bounded linear map from $N_{l-1}$ to $H^s_0(\Omega)$ given by $Tv = v_\mu$ with $\mu \ge \mu_0$ sufficiently large.

First of all, we provide some preliminary results.

\begin{lemma} \label{Lemma 7}
There exist constants $c_{29}, \dots, c_{32} > 0$ such that for all $\mu \ge \mu_0$ and $v \in N_{l-1} \cap S$,
\begin{gather}
\label{56} \int_{\R^N\times \R^N} \frac{|(v_\mu - v)(x) - (v_\mu - v)(y)|^2}{|x - y|^{N+2s}}\, dx dy \le c_{29}\, \mu^{-(N-2s)},\\[15pt]
\label{57} \int_{\R^N\times \R^N} \frac{|v_\mu(x) - v_\mu(y)|^2}{|x - y|^{N+2s}}\, dx dy \le \int_{\R^N\times \R^N} \frac{|v(x) - v(y)|^2}{|x - y|^{N+2s}}\, dx dy + c_{30}\, \mu^{-(N-2s)},\\[15pt]
\label{57.1} \int_\Omega |v_\mu|^{2_s^\ast}\, dx \ge \int_\Omega |v|^{2_s^\ast}\, dx - c_{31}\, \mu^{-N},\\[15pt]
\label{59} \int_\Omega \left[a\, (v_\mu^-)^2 + b\, (v_\mu^+)^2\right] dx \ge \int_\Omega \left[a\, (v^-)^2 + b\, (v^+)^2\right] dx - c_{32}\, \mu^{-N}.
\end{gather}
In particular, for all $\tau \ge 0$,
\[
\mathcal E(\tau v_\mu) \le \half\, \Big(I(v,a,b) + c_{33}\, \mu^{-(N-2s)}\Big)\, \tau^2 - \frac{1}{2_s^\ast} \left(\int_\Omega |v|^{2_s^\ast}\, dx - c_{31}\, \mu^{-N}\right) \tau^{2_s^\ast}
\]
for some constant $c_{33} > 0$.
\end{lemma}

\begin{proof}
First we note that since $N_{l-1} \cap S$ is bounded in $H^s_0(\Omega)$ and contained in the finite dimensional subspace $N_{l-1}$, which in turn is contained in $C^s(\R^N) \cap C^2(\Omega)$ by Lemma \ref{Lemma 3.1}, $N_{l-1} \cap S$ is bounded in $C^1(\closure{B_{2/\mu_0}(x_0)}) \cap C(\closure{\Omega})$.

Since $v_\mu = v$ outside $B_{1/\mu}(x_0)$,
\begin{multline*}
\int_{\R^N\times \R^N} \frac{|(v_\mu - v)(x) - (v_\mu - v)(y)|^2}{|x - y|^{N+2s}}\, dx dy \le \int_{B_{2/\mu} \times B_{2/\mu}} \hspace{-6.15pt} \frac{|(v_\mu - v)(x) - (v_\mu - v)(y)|^2}{|x - y|^{N+2s}}\, dx dy\\[7.5pt]
+ 2 \int_{B_{1/\mu} \times B_{2/\mu}^c} \frac{|(v_\mu - v)(x)|^2}{|x - y|^{N+2s}}\, dx dy =: I_1 + 2 I_2.
\end{multline*}
For $(x,y) \in B_{2/\mu}(x_0) \times B_{2/\mu}(x_0)$, we get
\[
|(v_\mu - v)(x) - (v_\mu - v)(y)| \le \bigg(\sup_{B_{2/\mu}(x_0)}\, |\nabla (v_\mu - v)|\bigg)\, |x - y|.
\]
We have
\[
\nabla (v_\mu - v)(x) = (\eta(\mu\, |x - x_0|) - 1)\, \nabla v(x) + \mu\, \eta'(\mu\, |x - x_0|)\, \frac{x - x_0}{|x - x_0|}\, v(x)
\]
and hence
\[
|\nabla (v_\mu - v)(x)| \le |\nabla v(x)| + 5 \mu\, |v(x)| \le \left(\frac{|\nabla v(x)|}{\mu_0} + 5\, |v(x)|\right) \mu \le c_{34}\, \mu \quad \forall x \in B_{2/\mu}(x_0)
\]
for some constant $c_{34} > 0$. Then
\[
I_1 \le c_{34}^2\, \mu^2 \int_{B_{2/\mu} \times B_{2/\mu}} \frac{dx dy}{|x - y|^{N-2\,(1-s)}} \le c_{34}^2\, \mu^2 \int_{B_{2/\mu} \times B_{4/\mu}} \frac{dx dz}{|z|^{N-2\,(1-s)}} = c_{35}\, \mu^{-(N-2s)}
\]
for some constant $c_{35} > 0$.

On the other hand, for $(x,y) \in B_{1/\mu}(x_0) \times B_{2/\mu}(x_0)^c$,
\[
|(v_\mu - v)(x)| \le |v(x)| \le c_{36}
\]
for some constant $c_{36} > 0$ and $|x - y| \ge |y| - |x| > 1/\mu$, so
\[
I_2 \le c_{36}^2 \int_{B_{1/\mu} \times B_{1/\mu}^c} \frac{dx dz}{|z|^{N+2s}} = c_{37}\, \mu^{-(N-2s)}
\]
for some constant $c_{37} > 0$. So \eqref{56} follows.

Proof of \eqref{57} is similar and the proofs of \eqref{57.1} and \eqref{59} are straightforward. Hence, Lemma~\ref{Lemma 7} is proved.
\end{proof}

\begin{lemma} \label{Lemma 8}
We have
\begin{equation} \label{63}
\sup_{v \in N_{l-1} \cap S,\, \sigma, \tau \ge 0}\, \mathcal E(\tau v_\mu + \sigma u_{\eps,\,\mu}) < \frac{s}{N}\, S_{N,\,s}^{N/2s}
\end{equation}
for all sufficiently small $\eps > 0$ and sufficiently large $\mu \ge \mu_0$.
\end{lemma}

\begin{proof}
For $v \in N_{l-1} \cap S$ and $\sigma, \tau \ge 0$,
\begin{equation} \label{61}
\mathcal E(\tau v_\mu + \sigma u_{\eps,\,\mu}) = \mathcal E(\tau v_\mu) + \mathcal E(\sigma u_{\eps,\,\mu}) - 2 \tau \sigma \int_{B_{1/2 \mu} \times B_{3/4 \mu}^c} \frac{u_{\eps,\,\mu}(x)\, v_\mu(y)}{|x - y|^{N+2s}}\, dx dy
\end{equation}
since $u_{\eps,\,\mu} = 0$ outside $B_{1/2 \mu}(x_0)$ and $v_\mu = 0$ in $B_{3/4 \mu}(x_0)$.

For $(x,y) \in B_{1/2 \mu}(x_0) \times B_{3/4 \mu}(x_0)^c$,
\[
|v_\mu(y)| \le |v(y)| \le c_{38}
\]
for some constant $c_{38} > 0$ since $N_{l-1} \cap S$ is bounded in $C(\closure{\Omega})$ as in the proof of Lemma \ref{Lemma 7}, and $|x - y| \ge |y| - |x| > 1/4 \mu$, so
\[
\abs{\int_{B_{1/2 \mu} \times B_{3/4 \mu}^c} \frac{u_{\eps,\,\mu}(x)\, v_\mu(y)}{|x - y|^{N+2s}}\, dx dy} \le c_{38} \int_{B_{1/2 \mu} \times B_{1/4 \mu}^c} \frac{u_{\eps,\,\mu}(x)}{|z|^{N+2s}}\, dx dz \le c_{39}\, \eps^{(N-2s)/2}
\]
for some constant $c_{39} > 0$ by \eqref{35}. So \eqref{61} gives
\begin{equation} \label{61.1}
\mathcal E(\tau v_\mu + \sigma u_{\eps,\,\mu}) \le \mathcal E(\tau v_\mu) + c_{39}\, \mu^{-(N-2s)}\, \tau^2 + \mathcal E(\sigma u_{\eps,\,\mu}) + c_{39}\, (\mu \eps)^{N-2s}\, \sigma^2.
\end{equation}

Since $a \le b$ and $(a,b) \in Q_l$,
\[
\mathcal I(v,a,b) = 1 - \int_\Omega \left[a\, (v^-)^2 + b\, (v^+)^2\right] dx \le 1 - a \int_\Omega v^2\, dx \le 1 - \frac{a}{\lambda_{l-1}} < 0.
\]
Moreover,
\[
\frac{1}{\lambda_{l-1}} \le \int_\Omega v^2\, dx \le \vol{\Omega}^{2s/N} \left(\int_\Omega |v|^{2_s^\ast}\, dx\right)^{2/2_s^\ast}
\]
and hence
\[
\inf_{v \in N_{l-1} \cap S}\, \int_\Omega |v|^{2_s^\ast}\, dx > 0.
\]
So it follows from Lemma \ref{Lemma 7} that
\begin{equation}\label{addl}
\mathcal E(\tau v_\mu) + c_{39}\, \mu^{-(N-2s)}\, \tau^2 \le 0
\end{equation}
for all sufficiently large $\mu \ge \mu_0$.

Then \eqref{61.1} and \eqref{addl} together with \eqref{37}--\eqref{38} gives for all sufficiently small $\eps > 0$,
\[
\mathcal E(\tau v_\mu + \sigma u_{\eps,\,\mu}) \le \begin{cases}
\left[\dhalf \left(1 - c_{40}\, \eps^{2s}\right) \sigma^2 - \dfrac{1}{2_s^\ast} \left(1 - c_{41}\, \eps^N\right) \sigma^{2_s^\ast}\right] S_{N,\,s}^{N/2s} & \text{if } N > 4s\\[15pt]
\left[\dhalf \left(1 - c_{40}\, \eps^{2s} \abs{\log \eps}\right) \sigma^2 - \dfrac{1}{4} \left(1 - c_{41}\, \eps^{4s}\right) \sigma^4\right] S_{4s,\,s}^2 & \text{if } N = 4s
\end{cases}
\]
for some constants $c_{40}, c_{41} > 0$ depending on $\mu$. Maximizing the right-hand side over all $\sigma \ge 0$ then gives
\[
\mathcal E(\tau v_\mu + \sigma u_{\eps,\,\mu}) \le \begin{cases}
\dfrac{s}{N}\, \dfrac{\left(1 - c_{40}\, \eps^{2s}\right)^{N/2s}}{\left(1 - c_{41}\, \eps^N\right)^{(N-2s)/2s}}\, S_{N,\,s}^{N/2s} & \text{if } N > 4s\\[15pt]
\dfrac{1}{4}\, \dfrac{\left(1 - c_{40}\, \eps^{2s} \abs{\log \eps}\right)^2}{1 - c_{41}\, \eps^{4s}}\, S_{4s,\,s}^2 & \text{if } N = 4s,
\end{cases}
\]
from which the desired conclusion follows.
\end{proof}

We are now ready to prove Theorem \ref{Theorem 2}.

\begin{proof}[Proof of Theorem \ref{Theorem 2}]
Since $b < \nu_{l-1}(a)$ and $\nu_{l-1}$ is continuous,
\[
b/(1 - \delta) \le \nu_{l-1}(a/(1 - \delta))
\]
if $\delta \in (0,1 - \max \set{a,b}/\lambda_{l+1})$ is sufficiently small. Then
\[
n_{l-1}(a/(1 - \delta),b/(1 - \delta)) \ge 0
\]
(see \eqref{29}) and hence
\begin{equation} \label{26}
\mathcal I(\theta(w,a/(1 - \delta),b/(1 - \delta)) + w,a/(1 - \delta),b/(1 - \delta)) \ge 0 \quad \forall w \in M_{l-1}
\end{equation}
(see \eqref{27}). For $\rho > 0$ and $w \in M_{l-1} \cap S_\rho$, set $u = \theta(w,a/(1 - \delta),b/(1 - \delta)) + w$. Then
\[
\mathcal E(u) = \frac{\delta}{2} \norm{u}^2 + \frac{1 - \delta}{2} \left(\norm{u}^2 - \int_\Omega \left[\frac{a}{1 - \delta}\, (u^-)^2 + \frac{b}{1 - \delta}\, (u^+)^2\right] dx\right) - \frac{1}{2_s^\ast} \int_\Omega |u|^{2_s^\ast}\, dx,
\]
and the quantity inside the parentheses is equal to $\mathcal I(u,a/(1 - \delta),b/(1 - \delta))$ and hence nonnegative by \eqref{26}, so
\[
\mathcal E(u) \ge \frac{\delta}{2} \norm{u}^2 + \o(\norm{u}^2) \text{ as } \norm{u} \to 0.
\]
Since
\[
\norm{u}^2 = \norm{\theta(w,a/(1 - \delta),b/(1 - \delta))}^2 + \norm{w}^2
\]
and $\norm{\theta(w,a/(1 - \delta),b/(1 - \delta))} = \O(\norm{w})$ by positive homogeneity (see \cite[Proposition 4.3.1]{MR3012848}), it follows that
\[
\mathcal E(u) \ge \frac{\delta}{2}\, \rho^2 + \o(\rho^2) \text{ as } \rho \to 0.
\]
So if $\rho > 0$ is sufficiently small,
\begin{equation} \label{10031}
\inf_{u \in A}\, \mathcal E(u) > 0,
\end{equation}
where $A = \set{\theta(w,a/(1 - \delta),b/(1 - \delta)) + w : w \in M_{l-1} \cap S_\rho}$.

Now we apply Theorem \ref{Theorem 3} taking $M = M_{l-1}$, $N = N_{l-1}$, $\theta = \theta(\cdot,a/(1 - \delta),b/(1 - \delta))$, $T : N_{l-1} \to H^s_0(\Omega)$ to be the bounded linear map given by
\[
Tv = v_\mu, \quad v \in N_{l-1},
\]
and $e = u_{\eps,\,\mu}$, with $\eps > 0$ sufficiently small and $\mu \ge \mu_0$ sufficiently large.

Since $u_{\eps,\,\mu}$ and functions in $T(N_{l-1})$ have disjoint supports, $u_{\eps,\,\mu} \in H^s_0(\Omega) \setminus T(N_{l-1})$. We have
\[
\norm{I_{N_{l-1}} - T} = \sup_{v \in N_{l-1} \cap S} \left(\int_{\R^N\times \R^N} \frac{|(v_\mu - v)(x) - (v_\mu - v)(y)|^2}{|x - y|^{N+2s}}\, dx dy\right)^{1/2} \to 0 \text{ as } \mu \to \infty
\]
by \eqref{56}, where $I_{N_{l-1}}$ is the identity map on $N_{l-1}$. If $\mu \ge \mu_0$ is sufficiently large, $\mathcal E(\tau v_\mu) \le 0$ for all $v \in N_{l-1} \cap S$ and $\tau \ge 0$ as in the proof of Lemma \ref{Lemma 8}, so
\[
\sup_{u \in T(N_{l-1})}\, \mathcal E(u) = 0.
\]
Together with \eqref{10031}, this gives the first inequality in \eqref{7}.

The second inequality in \eqref{7} also holds by Lemma \ref{Lemma 8}, so \eqref{8} together with \eqref{10031} and \eqref{63} gives
\[
0 < c := \inf_{h \in \mathcal{\widetilde{H}}}\, \sup_{u \in h(Q)}\, \mathcal E(u) < \frac{s}{N}\, S_{N,\,s}^{N/2s},
\]
where here $\mathcal{\widetilde{H}} = \{h \in \mathcal H : \restr{h}{T(N_{l-1})} = \id\}$ and $\mathcal H$ denotes the class of homeomorphisms $h$ of $H^s_0(\Omega)$ onto itself such that $h$ and $h^{-1}$ map bounded sets into bounded sets.

If $\mathcal E$ satisfies the \PS{c} condition, then Theorem \ref{Theorem 3} gives a critical point of $\mathcal E$ at the level $c$, which is nontrivial since $c > 0$. If, on the other hand, $\mathcal E$ does not satisfy the \PS{c} condition, then $\mathcal E$ has a \PS{c} sequence without a convergent subsequence, which then has a subsequence that converges weakly to a nontrivial critical point of $\mathcal E$ (see Gazzola and Ruf \cite[Lemma 1]{MR1441856}). In both cases Theorem~\ref{Theorem 2} is proved.
\end{proof}

\bigskip

\noindent{\bf Acknowledgement}

\medskip

\noindent The first, third and fourth authors are members of the
Gruppo Nazionale per l'Analisi Matematica, la Probabilit\`a e le loro Applicazioni
(GNAMPA) of the Istituto Nazionale di Alta Matematica (INdAM). The second author was partially supported by the Simons Foundation grant 962241. The fourth author was partially supported by MIUR--PRIN project ``Qualitative and quantitative aspects of nonlinear PDEs'' (2017JPCAPN\underline{\ }005).

\def\cdprime{$''$}

\end{document}